\documentclass{article}
\usepackage[utf8]{inputenc}
\usepackage{amssymb,amsmath,amsthm,amscd}
\usepackage{mathrsfs,mathtools}
\usepackage{enumerate}
\usepackage{graphicx}
\usepackage{color}

\usepackage{lmodern} 

  \newcommand{\C}{\ensuremath{\mathbf{C}}}%
  \newcommand{\R}{\ensuremath{\mathbf{R}}}%
  \newcommand{\N}{\ensuremath{\mathbf{N}}}%
  \newcommand{\Z}{\ensuremath{\mathbf{Z}}}%
  \newcommand{\T}{\ensuremath{\mathbf{T}}}%
  \newcommand{\Q}{\ensuremath{\mathbf{Q}}}

\newtheorem{thm}{Theorem}[section]

\newtheorem{lemma}[thm]{Lemma}
\newtheorem{corollary}[thm]{Corollary}
\newtheorem{prop}[thm]{Proposition}

\theoremstyle{remark}

\newtheorem{rem}[thm]{Remark}
\newtheorem{example}[thm]{Example}

\theoremstyle{definition}

\author{Tao Mei\footnote{Research partially supported by the NSF grant DMS-1266042.}\ \ and Mikael de la Salle\footnote{Research partially supported by the ANR projects NEUMANN and OSQPI.}
}
\title{Complete boundedness of Heat Semigroups on von Neumann Algebra of hyperbolic groups}

\begin{document}\maketitle

\begin{abstract} We prove that $(\lambda_g\mapsto e^{-t|g|^r}\lambda_g)_{t>0}$ defines a completely bounded semigroup of multipliers on the von Neuman algebra of hyperbolic groups for all real number $r$. One
ingredient in the proof is the observation that a construction of
Ozawa allows to characterize the radial multipliers that are bounded
on every hyperbolic graph, partially generalizing results of
Haagerup--Steenstrup--Szwarc and Wysocza\'nski. Another ingredient is an upper estimate of trace class norms for Hankel matrices, which is based on Peller's characterization of such norms.
\end{abstract}
\section{Introduction}
\noindent Let $\Delta=\partial^2 \theta$ be the Laplacian operator on the unit circle ${\T}$. The maps $$T^r_t=e^{-t(-\Delta)^{^\frac r2}}: e^{i2\pi k\theta}\rightarrow e^{-4\pi^2t|k|^r}e^{i2\pi k\theta}$$ define a semigroup of uniformly bounded operators on $L^\infty(\T)$ for any $0<r<\infty$. Moreover, for $0<r\leq 2$, $T_t^r$ are positivity preserving contractions, which can be easily seen by their integral representations. For $r=1,2$, the semigroups $T_t^r$ are called Poisson semigroup and heat semigroup. They both play important roles, and very often complementary roles, in harmonic analysis. It is easier to work with heat semigroups, for some problems, because of the general gaussian upper estimate of the heat kernels.

Let ${\mathbb F}_n$, $2\leq n\leq\infty$ be the free group on $n$ generators. Let $\lambda_g$ be the left regular representations of $g\in {\mathbb F}_n$. 
One may consider the analogue of the classical Poisson or heat semigroups on free group von Neumann algebra ${\mathcal L}({\mathbb F}_n)$, 
$$S_t^r: \lambda_g\rightarrow e^{-t|g|^r}\lambda_g, $$
with $|g|$ the reduced word length of $g$. U. Haagerup proved (see \cite{H79}) that for $r=1$, $(S_t^1)_{t\geq 0}$ is a semigroup of completely positive operators on ${\mathcal L}({\mathbb F}_n)$. For $0<r\leq 1$, the maps $S_t^r$ are therefore still unital completely positive (u. c. p.) by the theory of subordinated semigroups (see e.g. \cite{Y80}). For $r>1$, $S_t^r$ cannot be completely positive for all $t$, as the functions $\phi_t(g)=e^{-t|g|^r}$ is in $\ell_1({\mathbb F}_2)$ for each $t$, and the positive definiteness of $\phi_t$ would imply the amenability of the free group ${\mathbb F}_2$ (see the proof of \cite[Theorem 2.6.8]{BN08}).

This discussion settles the question of the complete positivity of the semigroup $S_t^r$. How about the complete boundedness (c. b.) of $S_t^r$? For $r\leq1$, $S_t^r$ have of course completely bounded norm $1$, since u. c. p. implies c. b. with norm one on $C^*$ algebras. For $r>1$, it is not hard to see that $S_t^r$ are bounded with an upper bound $1+c t^{-\frac 1{ r}}$ for each $t>0$ by the fact that the projection on words of length at most $k$ has c. b. norm $\simeq k$. The question is then wether $S_t^r$ are (completely) bounded uniformly in $t$ for any (all) $r>1$. By duality, the principle of uniform boundness and the pointwise convergence of $S_t^r$, asking whether $\sup_t \|S_t^r\| <\infty$ is equivalent to asking whether $S_t^r$, as an operator on the predual $L_1(\hat{\mathbb F}_n)$ of $\mathcal L({\mathbb F}_n)$, converges to the identity in the strong (or weak) operator topology when $t\rightarrow 0$. Similarly, $\sup_t \|S_t^r\|_{cb} <\infty$ if and only if $S_t^r$ converges to the identity in the stable point-norm (or stable point-weak) topology.

A characterization of the completely bounded radial multipliers on free groups was given by Haagerup and Szwarc in terms of some Hankel matrix being of trace class (this was published in \cite{HSS10}, see also \cite{W95}). On the other hand a famous result of V. V. Peller (\cite{P80}) states that a Hankel matrix belongs to the trace class iff the symbol function belongs to the Besov space $B_1^1$ (see Section 3). 
These results together provide a precise method for estimating the complete bounds of radial maps on free groups. However, the answer to the uniform complete boundedness of $S_t^r$ (for $r>1)$ remained open, to the best knowledges of the authors, before the writing of this article.
Besides the possible technicality of the corresponding estimation, another reason may be the general belief of a negative answer to the question. For example Knudby proved in \cite[Theorem 1.6]{K14} that the symbol of a completely contractive semigroup of radial multipliers on $\mathbb{F}_n$ ($n \geq 2$) grows linearly. By considering the function $n^r -C$ which does not grow linearly if $r>1$, this implies that for each $r>1$ there does not exist a constant $C$ such $\|S_t^r\|_{cb} \leq e^{t C}$ for all $t>0$. 

In the present work we record a proof that $S_t^r$ is a weak * continuous analytic semigroup of completely bounded maps on the free group von Neumann algebras for any $r>0$, by a careful estimation of the corresponding Besov norm. We remark that by \cite{W95,HM12,M14,D13}, this also proves that the analogous semigroups on free products of groups, on free products of operator algebras, and on amalgamated free products of finite von Neumann algebras, are completely bounded with bound independent on $t$. We do not elaborate on this and refer to \cite{W95,HM12,M14,D13} for the definitions of radial multipliers on free product of groups and (amalgamated) free products of operator algebras, and for details.

The authors hope that the study of $S_t^r$ would benefit the recent research in harmonic analysis on noncommutative $L_p$ spaces (see \cite{JMX06,JM12,JMP}), as the classical case suggests.

The main ingredient that we introduce is a result on the trace class norm of Hankel matrices with smooth symbol. A particular case of our result is the following (see Theorem \ref{R+} for a more precise statement).
\begin{thm}\label{thm=R+_intro} Let $f\colon [0,\infty) \to \R$ be a bounded continuous function of class $C^2$ on $(0,\infty)$, and $ \frac12\geq\alpha> 0$. Then the trace class norm of the matrix $(f(j+k) - f(j+k+1))_{j,k \geq 0}$ is less than $\frac{C}{\sqrt {\alpha}}\sqrt{\| x^{\frac 3 2 - \alpha} f''\|_{L^2(\R_+)} \| x^{\frac 3 2 + \alpha} f''\|_{L^2(\R_+)}}$ 
for some universal constant $C$.
\end{thm}

After we communicated to him the aforementioned result, Narutaka Ozawa asked us whether the same holds for other length functions on ${\mathbb F}_n$, or more generally for an arbitrary finitely generated hyperbolic group. To our surprise, the answer is yes. Indeed we can extend to hyperbolic graphs the sufficient condition from \cite{W95,HSS10} for a radial multiplier to be bounded. This is a consequence of \cite{O08}.
\begin{thm}\label{hyperbolicgraph}
Let $\Gamma$ be a hyperbolic graph with bounded degree. Then there is a constant $C$ such that for every $\dot{\phi}\colon\N \to \C$ such
that the matrix
\[ H=(\dot{\phi}(j+k)-\dot{\phi}(j+k+1))_{j,k \geq 0}\]
belongs to the trace class $S^1$, then $\lim_n \dot{\phi}(n)$ exists and the map
\[ A= (a_{x,y}) \in B(\ell^2 (\Gamma)) \mapsto
(\dot{\phi}(d(x,y))a_{x,y}) \in B(\ell^2 (\Gamma))\]
is bounded with norm less than $C \|H\|_{S^1} + \lim_n |\dot\phi(n)|$.
\end{thm}
In view of this theorem it is natural to wonder whether there are non hyperbolic groups also satisfying the same criterion. It turns out that there are, namely the groups $\Z^d$ ($d \geq 1$) for their standard generating set. This follows from some delicate estimates of $L^1(\R^d/\Z^d)$-norms that we prove in Section~\ref{section=Zd}.

The main result of this paper can be summarized as
\begin{thm}\label{main-result} Let $\Gamma$ be a finitely generated hyperbolic group, and $|\cdot|$ the word length associated to a fixed finite generating set. Then there is a constant $C$ such that for every $r>0$, and every $z \in \C$ with positive real part, the multiplier $S^r_z\colon\lambda_g \mapsto e^{-z |g|^r} \lambda_g$ is completely bounded on $\mathcal L(\Gamma)$ with norm less than 
\[C(1+|\tan(\arg z)|)^{3/2} (1+r).\]
\end{thm}
The dependence on the constant on the argument of $z$ is probably not optimal, but the order $r$ is sharp as $r \to \infty$ (see Example \ref{ex=heat_sgp_real}).

Our method applies to other radial multipliers. In Example \ref{Riesz} we show that the Bochner-Riesz multipliers $\lambda_g \mapsto (1-\frac{|g|^2}{N^2})^z\chi_{\{|g|\leq N\}} \lambda_g$ are completely bounded on the noncommutative $L^p$ spaces associated with ${\cal L}(\Gamma)$ uniformly in $N$ if   $|\frac 2p-1|< \mathrm{Re}(z)$, for any $1\leq p\leq \infty$ and any finitely generated hyperbolic group $\Gamma$. The same result holds for the F\'ejer-type multipliers  $\lambda_g \mapsto (1-\frac{|g|}{N})^z\chi_{\{|g|\leq N\}} \lambda_g$.


This article is organized as follows. In Section \ref{section=multipliers} we recall facts on Schur multipliers and prove Theorem \ref{hyperbolicgraph}. Section \ref{section=trace_class_estimates} contains estimates for the Schatten $1$-norm of Hankel matrices and the proof of Theorem \ref{main-result}. Section \ref{section=Zd} contains of proof that $\Z^d$ satisfies the conclusion of Theorem \ref{hyperbolicgraph}. In Section \ref{section=motivation}, we explain a motivation in studying $S_t^r$ and prove an end-point result on $H_\Sigma^\infty$-calculus.

{\bf Notation:} We denote by $\N$ the set of nonnegative integers including $0$~: $\N =\{0,1,2,\dots\}$. We denote by $S^1$ the Banach space of trace class operators on $\ell^2(\N)$. For a discrete group $\Gamma$ we denote by $L^p(\hat{\Gamma})$ the non commutative $L^p$ space on the von Neumann algebra of $\Gamma$.

\section{Radial multipliers on hyperbolic graphs}\label{section=multipliers}

\subsection{Reminders on Schur and Fourier multipliers}

We start with some reminders. If $X$ is a set, a function $\varphi\colon X \times X \to \C$ is a \emph{Schur multiplier} if the map $(a_{s,t})_{s,t \in X} \in B(\ell^2(X)) \mapsto (\varphi(s,t) a_{s,t})_{s,t \in X}$, denoted $M_\varphi$, is bounded on $B(\ell^2(X))$. The following proposition, which is essentially due to Grothendieck (\cite[Theorem 5.1]{Pis01}), characterizes the Schur multipliers.

\begin{prop}\label{Schur} Let $X$ be a nonempty set and assume that $\varphi: X\times X\mapsto \C$ and $C\geq0$ are given, then the following are equivalent:

(i) $\varphi(x,y)$ extends to a Schur multiplier on $B(\ell^2(X))$ with norm $\leq C$.

(ii) There exists a Hilbert space $H$ and two bounded, continuous maps $P,Q:\Gamma\mapsto H$ such that
$$\varphi(x,y)=\langle P(x),Q(y)\rangle$$ 
and $\|P\|_\infty\|Q\|_\infty\leq C$, where
$$\|P\|_\infty=\sup_{x\in X}\|P(x)\|, \|Q\|_\infty=\sup_{x\in X}\|Q(x)\|.$$
\end{prop}

Let $\Gamma$ be a discrete group and $\phi \colon \Gamma \to \C$. The \emph{Fourier multiplier} $\lambda_g\mapsto \phi(g)\lambda_g$ is completely bounded on the group von Neumann algebra $\mathcal L(\Gamma)$ if and only if the associated function $\varphi(g,h)=\phi(h^{-1}g)$ is a Schur multiplier. In this case the c.b. norm of the Fourier multiplier is equal to the norm (and also the c.b. norm) of the Schur multiplier (see \cite{BF84}). 

Let us state an immediate consequence of Proposition \ref{Schur} that will be used later (see \cite[Theorem 6.1]{Pis01} for a complete characterization of Hankelian Schur mutlipliers).
\begin{lemma}\label{lemma=HankelSchur} Let $(a_n)_{n \in \Z}$ be a finitely supported sequence of complex numbers. Then for each matrix $B = (b_{j,k})_{j,k \in \N} \in S^1$, 
\[ \|(a_{j+k} b_{j,k})_{j,k \in \N}\|_{S^1} \leq \left(\int_0^1 |\sum_{n \in \Z} a_n e^{2i\pi n \theta}| d\theta \right)\|B\|_{S^1}.\]
\end{lemma}
\begin{proof} Write $f(\theta) = \sum_{n \in \Z} a_n e^{2i\pi n \theta}$ and decompose $f= gh$ a product of two $L^2$ functions with $\|g\|_{L^2} \|h\|_{L^2} = \|f\|_{L^1}$. We can write $a_{j+k} = \int f(\theta) e^{-2i\pi(j+k)\theta} = \int g(\theta) e^{-2i\pi j\theta} h(\theta) e^{-2i\pi k\theta}  = \langle g e^{-2i\pi j\theta}, \overline h e^{2i\pi k\theta} \rangle$. Proposition \ref{Schur} implies that $(j,k) \mapsto a_{j+k}$ is a Schur multiplier with norm less than $\|g\|_{L^2}\|h\|_{L^2}=\|f\|_{L^1}$ on $B(\ell^2)$. By duality $M_{a_{j+k}}$ is bounded on $S^1$ with norm $\leq \|f\|_{L^1}$, which is the content of the Lemma.
\end{proof}
\subsection{Radial multipliers on trees} 

Let $\Gamma$ be a group with a fixed finite generating set, and denote by $d$ the associated left-invariant distance on $\Gamma$, or equivalently the distance on the Cayley graph of $\Gamma$. Recall that $d(x,y) = |x^{-1} y|$ where $|\cdot|$ is the word-length with respect to the fixed generating set of $\Gamma$.

Following \cite{HSS10}, we say a Schur multiplier $\varphi$ on $ \Gamma\times \Gamma$ is {\it radial} if $\varphi(x,y)=\dot{\phi}(d(x,y))$ for some function $\dot{\phi}:\N \mapsto \C$. We say a Fourier multiplier $\lambda_x \mapsto {\phi}(x)\lambda_x$ is a {\it radial Fourier multiplier} if $\phi(x)=\dot{\phi}(|x|)$ for some $\dot{\phi}\colon \N \to \C$. By \cite{BF84}, the completely bounded norm of a radial Fourier multiplier associated with $\phi$  equals to the norm of the Schur multiplier $\dot{\phi}(d(x,y))$.

An exact formula for the norm of radial Schur multipliers on free groups was given by Haagerup and Szwarc in 1987. The result was published in \cite{HSS10} with Steenstrup, where they extended the study to homogenous trees. A similar characterization for radial Fourier multipliers with respect to the block length on free products of groups is proved by Wysocza\'nski in 1995 (\cite{W95}). These results were recently extended to (amalgamated) free products of operator algebras in \cite{HM12,M14,D13}. We rewrite a special version of the results from \cite{HSS10} below.

\begin{thm}\label{HSS}(Haagerup, Steenstrup, Szwarc)
Let $X$ be a homogenous tree with degree $\geq3$ and $\dot{\phi}:\N\rightarrow \C$. Then the function $\varphi(x,y) = \dot{\phi}(d(x,y))$ is a Schur multiplier if and only if the matrix
\[ H=(\dot{\phi}(j+k)-\dot{\phi}(j+k+2))_{j,k \geq 0}\]
belongs to the trace class. In that case $\lim_{n} \dot{\phi}(2n)$ and $\lim_{n} \dot{\phi}(2n+1)$ exist, and 
$$\|M_{\dot{\phi}}\| \leq \lim_{n\rightarrow\infty}(| \dot{\phi}(2n)|+|\dot{\phi}(2n+1)|)+\|H\|_{S^1}.$$
\end{thm} 

In this section, we remark that a similar result to the ``if" part of Theorem \ref{HSS} holds on hyperbolic graphs. We first recall a proof of the ``if" part of Theorem \ref{HSS}, valid for all trees, and that we will later adapt for general hyperbolic graphs. Fix an infinite geodesic path $p$. For every $x \in X$, there is a unique infinite geodesic path $p_x$ which starts at $x$ and flows into $p$. Remark that for every $x,y \in X$ the geodesics $p_x$ and $p_y$ first meet on a point of the geodesic segment between $x$ and $y$, and then coincide forever. In formulas, the set $\{(i,j) \in \N \times \N, p_x(i) = p_x(j)\}$ is of the form $\{(i_0+k,j_0+k), k \in \N\}$, for $i_0,j_0$ satisfying $i_0+j_0 = d(x,y)$.

Since $H=(\dot{\phi}(i+j)-\dot{\phi}(i+j+2))_{0\leq i,j<\infty}$ belongs to the trace class, we can write $H=A^* B$ where $A,B$ are
Hilbert-Schmidt operators satisfying $\|H\|_{S^1}= \|A\|_{S_2}
\|B\|_{S_2}$. For $x,y \in X$, set
\[P(x) = \sum_{i \geq 0} B(e_i) \otimes \delta_{p_x(i)}
\in \ell^2(\N) \otimes \ell^2(X)\]
and 
\[Q(y) = \sum_{j \geq 0} A(e_j) \otimes \delta_{p_y(j)}
\in \ell^2(\N) \otimes \ell^2(X)\]
where $(e_i)_{i \geq 0}$ and $\delta_x$ are the coordinate orthonormal basis of
$\ell^2(\N)$ and $\ell^2({X})$.
We see that
$ \|P(x)\|^2 = \sum_{i \geq 0} \| B(e_i) \|^2=\|B\|_{S_2}^2$
and $ \|Q(y)\|^2 = \|A\|_{S_2}^2.$
Using that $\langle B e_i,A e_j\rangle = \dot{\phi}(i+j) - \dot{\phi}(i+j+2)$ we can
write
\begin{multline*} \langle P(x),Q(y) \rangle= \sum_{i,j, p_x(i)=p_y(j)} \dot{\phi}(i+j) - \dot{\phi}(i+j+2)\\ =\sum_{k=0}^\infty\dot{\phi}(d(x,y)+2k) - \dot{\phi}(d(x,y)+2k+2) \end{multline*}
which equals to $\dot{\phi}(d(x,y))-\lim_{n\rightarrow\infty}\dot{\phi}(2n)$ for $d(x,y)$ even and equals to $\dot{\phi}(d(x,y))-\lim_{n\rightarrow\infty}\dot{\phi}(2n+1)$ for $d(x,y)$ odd. 
Therefore,
 $$\dot{\phi}(d(x,y))=\langle P(x),Q(y) \rangle+\frac{1+(-1)^{d(x,y)}}2\lim_{n}\dot{\phi}(2n)+\frac{1-(-1)^{d(x,y)}}2\lim_{n}\dot{\phi}(2n+1) .$$
Fix a distinguished point $e\in X$. Note $(-1)^{d(x,y)}=(-1)^{d(x,e)}(-1)^{d(y,e)}$ since $d(x,e)+d(y,e)-d(x,y)$ is even. So $1\pm (-1)^{d(x,y)}$ is a Schur multiplier with norm $\leq 2$.
 We then obtain by Proposition \ref{Schur} that $$\|M_{\dot{\phi}}\|\leq \lim_{n\rightarrow\infty}(| \dot{\phi}(2n)|+|\dot{\phi}(2n+1)|)+\|H\|_{S^1}.$$

\subsection{Generalization to hyperbolic graphs}

Identify a connected graph $\Gamma$ with its vertex set and equip it with the graph distance $d$. A geodesic path $p$ is a finite or infinite sequence of points $p(0),...,p(k)...\in\Gamma$ such that $d(p(m), p(n)) = |m-n|$ for every $m, n$. A connected graph $\Gamma$ is {\it hyperbolic} if there exists a constant $\delta \geq 0$ such that for every geodesic triangle each edge is contained in the $\delta$-neighborhood of the union of the other two. A finitely generated group $\Gamma$ is {\it hyperbolic} if its Cayley graph is hyperbolic. The hyperbolicity property of a group is independent of the choice of generating set. A tree is a hyperbolic graph for $\delta=0$. See e.g. \cite{BN08}, section 5.3 for more information on hyperbolic groups.

To obtain an extension on general hyperbolic graphs, we need the following result of
Ozawa \cite[Proposition 10]{O08}.
\begin{prop}[Ozawa]\label{ozawa} Let $\Gamma$ be a hyperbolic graph with bounded
 degree. There is $C_0 \in \R$, a Hilbert space $\mathcal H$ and maps
 $\eta_i^{\pm}\colon \Gamma \to \mathcal H$ (for $i \in \N$) such that
\begin{enumerate}
\item\label{orthogonality} $\eta_i^\pm(x) \perp \eta_{j}^\pm(x)$ for all $x \in \Gamma$ and
 $|i-j|\geq 2$.
\item\label{bound} $\|\eta_i^\pm(x)\| \leq \sqrt{C_0}$ for all $i \in \N$ and $x \in \Gamma$,
\item\label{scalprod} $\sum_{i+j=n} \langle \eta_i^+(x), \eta_j^-(y)\rangle =
 \left\{ \begin{array}{cc} 1 &\textrm{if }d(x,y)\leq n\\ 0 & \textrm{otherwise}\end{array}\right.$.
\end{enumerate}
\end{prop}
The $\eta_i^{\pm}(x)$ provided by this Proposition will play the role of the vectors $\delta_{p_x(i)} \in \ell^2(X)$ in the preceding proof.

 Assume $H=(\dot{\phi}(i+j)-\dot{\phi}(i+j+1))_{0\leq i,j<\infty}\in S^1$. Then the diagonal of $H$ belongs to $\ell^1$, and hence $\lim_n \dot\phi(n)$ exists. We proceed similarly. Write $H=A^* B$ with $\|H\|_{S^1}= \|A\|_{S_2}
\|B\|_{S_2}$. 
 Set
\[P(x) = \sum_{i \geq 0} B(e_i) \otimes \eta_i^+(x)
\in \ell^2 \otimes \mathcal H\] 
and 
\[Q(y) = \sum_{j \geq 0} A(e_j) \otimes \eta_j^-(x)
\in \ell^2 \otimes \mathcal H,\]for $x,y \in \Gamma$. 
From condition \ref{orthogonality} and \ref{bound} in Proposition \ref{ozawa} we see that
$|\langle B(e_i) \otimes \eta_i^+(x),B(e_j) \otimes \eta_j^+(x)\rangle|$ is
zero if $|i-j|\geq 2$, and always less than $C_0\frac{\|B(e_i)\|^2 +
 \|B(e_i)\|^2}{2}$. Hence 
\[ \|P(x)\|^2 = \sum_{i,j \geq 0} \langle B(e_i) \otimes
 \eta_i^+(x),B(e_j) \otimes \eta_j^+(x)\rangle \leq 3 C_0 \|B\|_{S_2}^2.\] 
 Similarly, $\sup_y \|Q(y)\|^3 \leq 3 C_0 \|A\|_{S_2}^2$. 
 We claim that $\langle P(x),Q(y)
\rangle= \dot{\phi}(d(x,y))$ for all $x,y \in \Gamma$. Indeed,
\[ \langle P(x),Q(y)
\rangle = \sum_{i,j} \langle B e_i \otimes \eta_i^+(x) , A e_j \otimes
\eta_i^-(y)\rangle.\]
Using that $\langle B e_i,A e_j\rangle = \dot{\phi}(i+j) - \dot{\phi}(i+j-1)$ we can
write
\[ \langle P(x),Q(y) = \sum_{n \geq 0} \left((\dot{\phi}(n) - \dot{\phi}(n+1)) \sum_{i+j=n}
\langle \eta_i^+(x), \eta_j^-(y)\rangle\right),\]
which equals $\dot{\phi}(d(x,y)) - \lim_n \dot{\phi}(n)$ from assumption \ref{scalprod} in Proposition
\ref{ozawa}. By Propostion \ref{Schur}, we obtain Theorem \ref{hyperbolicgraph}.

\begin{rem} In Theorem \ref{hyperbolicgraph}, we can take $C=1$ if $\Gamma$ is a tree. Otherwise, $C$ may depend on $\Gamma$.
\end{rem}
\begin{rem}\label{rem=difference-1-2} As the example of $\dot \phi(k) = (-1)^k$ shows, the condition that $(\dot{\phi}(j+k)-\dot{\phi}(j+k+1))_{j,k \geq 0}$ is of trace class is stronger than the condition in Theorem \ref{HSS} for $(\dot{\phi}(j+k)-\dot{\phi}(j+k+2))_{j,k \geq 0}$, but is necessary for general hyperbolic graphs. Indeed, if $\Gamma_0$ is the Cayley graph of $(\Z/3\Z) \ast (\Z/3\Z) \ast (\Z/3\Z)$ with generating set the union of the $3$ copies of $\Z/3\Z$, then by \cite[Theorem 6.1]{W95} the Schur multiplier with symbol $\dot \phi(d(x,y))$ is bounded on $B(\ell^2(\Gamma_0))$ if and only if $(\dot{\phi}(j+k)-\dot{\phi}(j+k+1))_{j,k \geq 0}$ belongs to the trace class. Theorem \ref{hyperbolicgraph} can therefore be read as ``a function $\phi\colon \N \to \C$ defines a bounded radial multiplier on every hyperbolic graph $\Gamma$ with bounded degree if and only if it defines a bounded radial multiplier on $\Gamma_0$''.
\end{rem}
\begin{rem} The fact that it is the matrix $(\dot{\phi}(j+k)-\dot{\phi}(j+k+2))_{j,k \geq 0}$ that appears in Theorem \ref{HSS} is related to the fact that trees are bipartite (a graph is bipartite is it does not contain any odd length cycle).
A modification (left to the reader) of the proofs in \cite{O08} actually shows that when the hyperbolic graph $\Gamma$ is bipartite, then Theorem \ref{hyperbolicgraph} holds also for $H$ replaced by $(\dot{\phi}(j+k)-\dot{\phi}(j+k+2))_{j,k \geq 0}$. In that case the statement becomes ``a function $\phi\colon \N \to \C$ defines a bounded radial multiplier on every bipartite hyperbolic graph $\Gamma$ if and only if it defines a bounded radial multiplier on ${\mathbb F}_2$ with standard generating set''.
\end{rem}
In particular, if we apply the preceding Theorem \ref{hyperbolicgraph} to a Cayley graph of a finitely generated hyperbolic group $\Gamma$ and recall that the word length of $y^{-1}x$ equals to $d(y^{-1}x,e)=d(x,y)$ on its Cayley graph, we get
\begin{corollary}\label{hyperbolicgroup}
Let $\Gamma$ be a finitely generated hyperbolic group, and $|\cdot|$ the length
function on $\Gamma$ associated to a finite generating set of
$\Gamma$. Then there is a constant $C \in \R$ such that if $\dot{\phi}\colon \N
\to \C$ is a function such that the infinite Hankel matrix $$H=(\dot{\phi}(k+j)-\dot{\phi}(k+j+1))_{0\leq k,j<\infty}$$
belongs to the trace class $S^1$, then $ \lambda_g\mapsto
\dot{\phi}(|g|)\lambda_g$ extends to a completely bounded map on ${\mathcal
 L}(\Gamma)$ with norm $\leq C \|H\|_{S^1} + \lim_{n\rightarrow \infty}|\dot{\phi}(n)|$.
\end{corollary}

\subsection{Weighted length functions on ${\mathbb F}_n$}
Corollary \ref{hyperbolicgroup} in particular applies to ${\mathbb F}_n$ equipped with the length function associated to other finite generating sets than the standard generating set.

One may also consider weighted lengths on a free group ${\mathbb F}_n$, $n \in \N \cup \{\infty\}$. Denote the free generators by $g_1,g_2,...g_k,...$. Fix a sequence of positive real numbers $a=(a_k)_k$, for $g=g_{i_1}^{k_1}g_{i_2}^{k_2}...g_{i_m}^{k_m}\in {\mathbb F}_n,i_j\neq i_{j+1}$, let
$$|g|_a=\sum_{j=1}^ma_{i_j}|k_j|.$$ 

Theorem \ref{main-result} also extends to 
the weighted lengths $|g|_a$ on free products groups ${\mathbb F}_n, 2\leq n\leq \infty$ because of the following Proposition.

For $\phi:{\R}_+\mapsto {\C}$, we denote by $\phi_t(\cdot)=\phi(t\cdot)$ and $m_{\phi^a}$ the Fourier multiplier sending $\lambda_g$ to $\phi(|g|_a)\lambda_g$. We ignore the ``a'' when $a_k=1$ for all $k$.
\begin{prop}\label{dilate}
Given a continuous function $\phi$, suppose $\|m_{\phi_t}\|_{c.b}<C$ on ${\mathbb F}_\infty$ for all $0<t<1$, then $\|m_{\phi_t^a}\|_{c.b}<C$ for the same constant $C$ for any sequence $a$ and any $t>0$. 
\end{prop}

\begin{proof} Assume first that $a_k\in {\N^*=\{1,2,\dots\}}$. Let $T_a$ be the trace preserving *-homomorphism sending $\lambda_{g_j}$ to $\lambda_{g_j^{a_j}}$. Then 
$$T_a \circ m_{\phi_t^a} =m_{\phi_t} \circ T_a,$$ 
which shows that $\|m_{\phi_t^a}\|_{c.b}< \|m_{\phi_t}\|_{c.b}$ since $T_a$ is completely isometric. 

Next we assume $a_k\in {\Q}_+$ with a common denominator $N$. Then
$$ T_{Na} \circ m_{\phi_t^a} = m_{\phi_{\frac tN}} T_{Na}.$$
Therefore, $m_{\phi_t^a}$ is completely bounded with upper bound $\sup_{0<t<1}\|m_{\phi_t}\|_{c.b}$.

The general case follows by approximation.
\end{proof}

\section{Complete boundedness of $S_t^r$}\label{section=trace_class_estimates}

Theorem \ref{HSS} (respectively Corollary \ref{hyperbolicgroup}) states that the completely bounded norm of the map $S_t^r$ on $\mathcal L({\mathbb F}_n)$ (respectively $\mathcal L(\Gamma)$ for a hyperbolic group $\Gamma$) is equivalent to (respectively dominated by) the trace class norm of the corresponding Hankel matrix. 

In this section we give an upper bound on the trace class norm of Hankel matrices with smooth symbol, that we apply to several explicit examples. In particular Theorem \ref{main-result} is an immediate consequence of Corollary \ref{hyperbolicgroup} and Examples \ref{ex=heat_sgp_complex} and \ref{ex=heat_sgp_real}.

We start by stating a more precise version of Theorem \ref{thm=R+_intro}.
 \begin{thm}\label{R+} Let $f\colon [0,\infty) \to \R$ be a bounded continuous function of class $C^2$ on $(0,\infty)$, and $ \frac12\geq\alpha> 0$. Then, for any $t>0$, the trace class norm of the matrix $f(t(j+k)) - f(t(j+k+1))$ satisfies the inequality
\begin{eqnarray}\label{tinv}
 \left\|\left(f(t(j+k))-f(t(j+k+1))\right)_{j,k \geq 0}\right\|_{S^1} 
 \leq \frac{C}{\sqrt {\alpha}} \sqrt{ A B} \leq \frac{2C}{\sqrt {\alpha}} B
 \end{eqnarray}
for some universal constant $C$, where 
\[ A = \sqrt{\| x^{\frac 1 2 - \alpha} f'\|_{L^2(\R_+)} \| x^{\frac 1 2 + \alpha} f'\|_{L^2(\R_+)}},\]
\[ B = \sqrt{\| x^{\frac 3 2 - \alpha} f''\|_{L^2(\R_+)} \| x^{\frac 3 2 + \alpha} f''\|_{L^2(\R_+)}}.\]
\end{thm}
We postpone the proof of Theorem \ref{R+} to the end of this section. 

\subsection{Examples of applications}
We give several applications of Theorem \ref{R+}.

\begin{example}[The semigroup $\lambda_g \mapsto (1+|g|)^{-z} \lambda_g$ on free groups] For every $z \in C$ with positive real part, the formula 
\[ (1+n)^{-z} = \frac{1}{\Gamma(z)} \int_0^\infty t^{z-1} e^{-t} e^{-tn} dt,\]
(where $\Gamma(z) = \int_0^\infty t^{z-1} e^{-t} dt$ is the Gamma function) together with the fact that $S_t^1\colon \lambda_g \mapsto e^{-t|g|} \lambda_g$ is unital completely positive for every $t>0$ implies that $\lambda_g \mapsto (1+|g|)^{-z} \lambda_g$ is unital completely positive if $z \in \R_+$, and unital completely bounded with completely bounded norm less than $\Gamma(\mathrm{Re}(z))/|\Gamma(z)|$ otherwise. This estimate is far from optimal: if $\mathrm{Re}(z) \geq 3$, the rapid decay property implies that this cb norm is less than $\sum_{n \geq 0} (1+n)^{-2}$, whereas $\Gamma(\mathrm{Re}(z))/|\Gamma(z)|$ is unbounded. Theorem \ref{R+} gives complementary estimates for this norm in the regime $\mathrm{Re}(z) \leq 3$. Namely let $z=a+ib$ with $0<a\leq 3$. Taking $f(x) = (1+x)^{-z}$ and $\alpha = \min(1,a)/2$ then 
\[ \|x^{\frac 1 2 \pm \alpha}f'\|_{L^2(0,\infty)} \leq \|(1+x)^{\frac 1 2 \pm \alpha}f'\|_{L^2(1,\infty)} = \frac{|z|}{\sqrt{2a \pm 2\alpha}} \leq \frac{|z|}{\sqrt a} \]
and similarly 
\[ \|x^{\frac 3 2 \pm \alpha}f''\|_{L^2(0,\infty)} \leq \frac{|z(1+z)|}{\sqrt{2a \pm 2\alpha}} \leq \frac{|z(1+z)|}{\sqrt{a}} .\]
Hence Theorem \ref{HSS} and Theorem \ref{R+} imply that the completely bounded norm of $\lambda_g \mapsto (1+|g|)^{-z} \lambda_g$ is less than $C\frac{|z|}{\mathrm{Re}(z)} \sqrt{|1+z|}$. This implies that for $\omega \in [0,\frac \pi 2)$, the cb norm of $\lambda_g \mapsto (1+|g|)^{-z} \lambda_g$ is bounded by $C(1+\tan \omega)^{\frac 3 2}$ on $\{z \in \C, |\arg z| \leq \omega\}$.

The same estimates also hold for the multiplier $\lambda_g \mapsto \max(1,|g|)^{-z}\lambda_g$. This follows from Theorem \ref{HSS} and from the following general inequality applied to $a_n = \max(1,n)^{-z} - (n+1)^{-z}$. If $a_n$ is any sequence of complex numbers,
\[ \|(a_{j+k+1})_{j,k \geq 0}\|_{S^1} \leq \|(a_{j+k})_{j,k \geq 0}\|_{S^1} \leq |a_0|+2 \|(a_{j+k+1})_{j,k \geq 0}\|_{S^1}\]
the first inequality if obvious, whereas the second is the triangle inequality for the trace norm $\| \cdot\|_{S^1}$ applied to the decomposition
\[a_{j+k} = a_{0} 1_{j=k=0} + a_{k}1_{j=0,k>0} + a_{j+k}1_{j>0}.\]

Note also that by Corollary \ref{hyperbolicgroup}, the same results hold for the semigroups $\lambda_g \mapsto (1+|g|)^{-z} \lambda_g$ and $\lambda_g \mapsto \max(1,|g|)^{-z}\lambda_g$ on every hyperbolic group, up to some multiplicative constant depending on the group.
\end{example}

\begin{example}[Fej\'er Kernel]\label{Fejer} Given $N\in {\N}$, let $F_N(k)=(1-\frac {k}{N})\chi_{[0,N]}(k)$. Define Fej\'er multiplier as
 $$m_{F_N}: \lambda_ g\mapsto F_N(|g|)\lambda_g.$$
 Then $\|m_{F_N}\|_{cb}\simeq \log N$ on free group von Neumann algebras. 
 In fact, 
\begin{equation} \label{eq=Fejer_unbounded}
\|(F_N(k+j)-F_N(k+j+1))_{k,j}\|_{S^1}=\|(\frac {1}{N})_{k+j\leq N-1}\|_{S^1}\simeq \log N.\end{equation}

If one applies Theorem \ref{R+} for the function $f(x) = (1-x)^\delta\chi_{[0,1]}(x)$, we see that 
$$\sup_N \|(F_N(k+j)^\delta-F_N(k+j+2)^\delta)_{k,j}\|_{S^1} <\infty$$
if $\delta> \frac 3 2$, because $f'' \in L^2$ if and only if $\delta>\frac 3 2$. We will see in Example \ref{Riesz} that the previous inequality actually holds for all $\delta>1$.


 
\end{example} 

 \begin{example}[Bochner-Riesz Mean]\label{Riesz} Bochner-Riesz mean is a ``smoothed" Fej\'er multiplier in modern harmonic analysis. Given $N\in {\N}$, let $B^\delta_N(k)=(1-\frac {k^2}{N^2})^\delta\chi_{[0,N]}(k)$ for $\delta\in {\C}.$ Define Bochner-Riesz multiplier as $$m^\delta_{B_N}: \lambda_g\mapsto B^\delta_N(|g|)\lambda_g.$$

As for the Fej\'er kernel, a direct application or Theorem \ref{R+} would yield that the Bochner-Riesz kernels are completely bounded on von Neumann algebras of hyperbolic groups if $\mathrm{Re}(\delta)>\frac 3 2$. However, using the known boundedness properties of Bochner-Riesz multipliers on $\mathcal L(\Z)$, one can decrease this to $\mathrm{Re}(\delta)>1$. 

Before that, we observe that, as in $\mathcal L(\Z^n)$, the problem of complete boudnedness of the Bochner-Riesz and Fej\'er multipliers on hyperbolic groups are equivalent, in the sense that for all $\mathrm{Re}(\delta)\geq 0$ and all hyperbolic groups, there is a constant $C$ such that for all $N$,
\begin{equation}\label{eq:equivalence_Fejer_BochnerRiesz} \frac{1}{C} \leq \frac{\|m^\delta_{B_N}\|_{cb(\mathcal L\Gamma)}}{\|m^\delta_{F_N}\|_{cb(\mathcal L\Gamma)}} \leq C.\end{equation}
Indeed, given $\mathrm{Re}(\delta)\geq 0$, let $f_\delta$ and $g_\delta$ be a compactly supported $C^2$ functions on $[0,\infty)$ satisfying $f_\delta(x) = 1/g_\delta(x) = (1+x)^\delta$ for all $x \in [0,1]$, so that $B^\delta_N (k) = f_\delta(k/N) F^\delta_N(k)$ and $F^\delta_N (k) = g_\delta(k/N) F^\delta_N(k)$. By Corollary \ref{R+} $\sup_N \| (f_\delta(\frac{i+j}{N}) - f_\delta(\frac{i+j+1}{N}) )_{j,k}\|_{S^1}<\infty$ and same for $g_\delta$. By Corollary \ref{hyperbolicgroup}, the multipliers corresponding to $f_\delta(|g|/N)$ and $g_\delta(|g|/N)$ are therefore bounded uniformly in $N$ on every hyperbolic group. This proves \eqref{eq:equivalence_Fejer_BochnerRiesz}.

We can now prove that Bochner-Riesz multipliers (and hence the Fej\'er multipliers by \eqref{eq:equivalence_Fejer_BochnerRiesz}) are completely bounded on every hyperbolic groups, and in particular on all free groups, if $\mathrm{Re}\delta>1$. 
This follows from Corollary \ref{hyperbolicgroup} and the estimate
\begin{equation}\label{eq:Bochner_Riesz_delta>1}
\forall \mathrm{Re}(\delta)>1, \sup_N \| ( B^\delta_N(j+k) - B^\delta_N(j+k+1) )_{j,k \geq 0}\|_{S^1} <\infty.
\end{equation}

Let us prove \eqref{eq:Bochner_Riesz_delta>1}. By derivating we can write
\[ B^\delta_N(k) - B^\delta_N(k+1) = \int_0^1 2 \delta (1-\frac{(k+t)^2}{N^2})^{\delta-1} \chi_{[0,N]}(k+t) \frac{k+t}{N^2} dt.\]
Denote by $B_{N,t}^{\delta-1}(k) = (1-\frac{(k+t)^2}{N^2})^{\delta-1} \chi_{[-N,N]}(k+t)$; for $t=0$ and $k \geq 0$ this is $B_N^{\delta -1}$. Let $f_1$ be a compactly supported $C^2$ function on $[0,\infty)$ such that $f(x) = -x^2/2$ on $[0,1]$. The previous equality becomes
\[ B^\delta_N(k) - B^\delta_N(k+1) = \int_0^1 2 \delta B_{N,t}^{\delta-1}(k) ( \frac{t-1/2}{N^2} + f_1(\frac{k}{N}) - f_1(\frac{k+1}{N}) ) dt.\]
The trace norm of the matrix $(B_{N,t}^{\delta-1}(j+k) \frac{t-1/2}{N^2} )_{j,k \geq 0}$ is less than the sum of the absolute values of its entries, which is less than $\frac 1 2$. Moreover by Corollary \ref{R+} the trace norm of the matrix $(f_1(\frac{j+k}{N}) - f_1(\frac{j+k+1}{N}))$ is bounded uniformly in $N$, by some constant $C$. Therefore Lemma \ref{lemma=HankelSchur} and the previous equality imply that the trace norm of $( B^\delta_N(j+k) - B^\delta_N(j+k+1) )_{j,k \geq 0}$ is less than
\begin{eqnarray*}
&&|\delta| + 2 C |\delta| \sup_{t\in [0,1]} \| \sum_{n \in \Z} B_{N,t}^{\delta - 1}(n) e^{2i\pi n \theta}\|_{L^1(\R/\Z)}\\
&\leq& |\delta| + 2 C |\delta| \|m_{B_N}^{\delta-1}\|_{cb(L^\infty(\R))}\leq |\delta| + 2 C |\delta| e^{C|Im \delta|^2}
\end{eqnarray*}
The first inequality follows by  embedding $L^\infty(\R/\Z)$ into $L^\infty(\R)$ via   $ \sum a_ne^{2i\pi n \theta}\mapsto  \sum a_ne^{2i\pi n (x-t)}$. The second inequality is quoted from \cite[Prop. 10.2.2]{G14}, and the constant $C$ depends only on $\mathrm{Re}(\delta)$. This proves \eqref{eq:Bochner_Riesz_delta>1} and 
\begin{eqnarray}\label{ConstantC}
\|m_{B_N}^{\delta}\|_{cb({\cal L}(\Gamma))}\leq e^{C+C|Im \delta|^2}
\end{eqnarray}
for $\mathrm{Re}(\delta)>1$ with $C$ only depends on $\mathrm{Re}(\delta)$ and $\Gamma$. 
We can observe that the assumption $\mathrm{Re}(\delta)>1$ is needed. Indeed, for $\delta=1$, \eqref{eq:Bochner_Riesz_delta>1} does not hold because of  \eqref{eq=Fejer_unbounded} and \eqref{eq:equivalence_Fejer_BochnerRiesz}, so $m_{B_N}^1$ is not c.b. on ${\cal L}({\mathbb F}_2)$ uniformly in $N$ by Theorem \ref{HSS}. 
 
 Fix  $0<\varepsilon<1$, let $C$ be the constant in \eqref{ConstantC}  such that the multiplier $F(z)=m_{B_N}^z e^{Cz^2-5C}$ is c.b. on ${\cal L}(\Gamma)$ uniformly on the complex line $\{z; \mathrm{Re}(z)=1+\varepsilon\}$. Note that $F(z)$    is c.b. on $L^2(\hat\Gamma)$ uniformly in $N$ on the imaginary line $\{z; \mathrm{Re}(z)=0\}$. By complex interpolation and duality, we  get $ m_{B_N}^\delta$ is  c.b. on $L^p(\hat \Gamma)$ uniformly in $N$ for any $|\frac2p-1|< \mathrm{Re}(\delta),1\leq p\leq\infty$. The same holds for $ m_{F_N}^\delta$ because of (\ref{eq:equivalence_Fejer_BochnerRiesz}).
\end{example}

\begin{example}\label{ex=heat_sgp_real} Given $r>0$, let $\alpha=\frac {\min\{r, 1\}}2$ and $f(x) = e^{-x^r}$. We then have, 
$$\|x^{\frac 1 2\pm\alpha} f'\|_{L^2(0,\infty)}\leq c \sqrt r, \|x^{\frac 3 2\pm\alpha}f''\|_{L^2(0,\infty)}\leq c(1+ r)\sqrt r$$
Applying Corollary \ref{R+}, we get
\[ \sup_{t\geq 0} \|(e^{-t(j+k)^r} - e^{-t(j+k+1)^r})_{j,k\geq 0}\|_{S^1}\leq c(1+r).\]
Moreover the order $r$ as $r$ goes to $\infty$ is optimal. Indeed, if $n$ is the integer part of $r$ and $t=n^{-r}$ using the inequality $\|A\|_{S^1} \geq \sum_{j=0}^n |A_{j,n-j}|$ we have
\[ \|(e^{-t(j+k)^r} - e^{-t(j+k+1)^r})_{j,k\geq 0}\|_{S^1} \geq (n+1)\left(e^{-1} - e^{-(1+1/n)^r}\right) \sim r(e^{-1} - e^{-e})\]
as $r \to \infty$.
\end{example}

\begin{example}\label{ex=heat_sgp_complex} For every $z=a+bi\in \C$ with $|arg z| \leq \omega<\frac \pi2$, let $f(x)=e^{-zx^r}$. Denote $K=(1+\tan^2 \omega)$. Then $|z|^2\leq Ka^2$ and 
\begin{eqnarray*}
 |f'|^2&=&|zrx^{r-1}e^{-zx^r}|^2\leq Ka^2r^2x^{2r-2}e^{-2ax^r} \\
 |f''|^2&=&|-z^2r^2x^{2r-2}e^{-zx^r}+zr(r-1)x^{r-2}e^{-zx^r}|^2\\
 &\leq& 2K^2 |a^2r^2x^{2r-2}e^{-ax^r}|^2+2K|a r(r-1)x^{ r-2}e^{-ax^r}|^2.
 \end{eqnarray*}
Setting $\alpha=\min\{\frac r2,1\}$ we then have using the change of variable $s= a x^r$,
\begin{eqnarray*}
 \|x^{\frac12\pm\alpha}f'\|_{L_2(\R_+)}^2&\leq& K\int_{\R_+}a^2r^2x^{2r\pm 2\alpha-1}e^{-2ax^r}dx\\
 &=&K\int_{\R_+}r a^{\mp\frac {2\alpha} r}s^{1 \pm \frac{2\alpha}{r}}e^{-2s}ds\simeq K a^{\mp\frac {2\alpha} r}r.\\
 \|x^{\frac32\pm\alpha}f''\|_{L_2(\R_+)}^2&\leq&2K^2\int_{\R_+} a^4r^4x^{4r\pm 2\alpha-1}e^{-2ax^r}dx\\ &&\hskip .5cm +2K\int_{\R_+} a^2r^2(r-1)^2x^{2r\pm 2\alpha-1}e^{-2ax^r}dx\\
 & = &2K^2 \int_{\R_+} a^{\mp\frac {2\alpha} r}r^3s^{3 \pm \frac{2\alpha}{r}}e^{-2s}ds\\ &&\hskip .5cm+2K\int_{\R_+} a^{\mp\frac {2\alpha} r }r(r-1)^2 s^{1 \pm \frac{2\alpha}{r}}e^{-2s}ds\\
 &\simeq&2K^2a^{\mp\frac {2\alpha} r}r^3+2K a^{\mp\frac {2\alpha} r }r(r-1)^2.
 \end{eqnarray*}
 Corollary \ref{R+} yields that
\[ \sup_{t\geq 0} \|(e^{-zt(j+k)^r} - e^{-zt(j+k+1)^r})_{j,k\geq 0}\|_{S^1}<c(1+(\tan \omega)^{\frac 3 2}) (1+ r).\]
\end{example} 



\subsection{The proof}\label{subsection=proof}
In the sequel we consider the unit circle $\T = \{e^{2i\pi t}, t \in \R/\Z\}$ equipped with the Lebesgue probability measure, and the unit disk $\mathbf D = \{z \in \C, |z|<1\}$ equipped with the Lebesgue probablity measure $\frac{dz}{\pi}$.

We now turn to the proof of Theorem \ref{R+}. The proof relies on Peller's characterization of trace class Hankel matrices \cite{P80} that we now recall. With the formulation given in \cite[Theorem 3.1]{HSS10} (which gives very good constants), Peller's theorem states that a Hankel matrix $(a_{j+k})_{j,k \geq 0}$ belongs to the trace class if and only if the function $g(z) = \sum_{n \geq 0} (n+1)(n+2) a_n z^n$ belongs to $L^1(\mathbf{D},\frac{dz}{\pi})$, and
\[ \frac \pi 8 \| g\|_{L^1(\mathbf{D},\frac{dz}{\pi})} \leq \|(a_{j+k})_{j,k\geq 0}\|_{S^1} \leq \|g\|_{L^1(\mathbf{D},\frac{dz}{\pi})}.\]
The condition $\sum_{n \geq 0} (n+1)(n+2) a_n z^n \in L^1(\mathbf{D})$ is one of the equivalent conditions for the series $\sum_{n \geq 0} a_n z^n$ to belong to the Besov space of analytic functions $B_1^1$. In the sequence we will work with another classical condition, that is more suited for our proof.

Consider the classical de la Vall\'ee Poussin kernels $(W_n)_{n \geq 0}$. The $W_n$'s are functions on $\T$ given by their Fourier coefficients. $W_0(z) = 1+z$ and for $n>0$
\[\widehat{W_n}(k)= \left\{ 
\begin{array}{ll}
2^{-n+1}(k-2^{n-1}) & \textrm{if }2^{n-1} \leq k \leq 2^n\\
2^{-n}(2^{n+1} - k) & \textrm{if }2^{n} \leq k \leq 2^{n+1}\\
0 & \textrm{otherwise.} 
\end{array}\right.\]

The Besov space $B_1^1$ of analytic functions is the Banach space of series $\varphi(z) = \sum_{n \geq 0} a_n z^n$ with $a_n \in \C$ such that 
\begin{equation}\label{eq:def_B11} \|\varphi\|_{B_1^1} = \sum_{n \geq 0} 2^n \|W_n \ast \varphi\|_{L^1(\T)} < \infty.\end{equation}
We refer to \cite{P03} for the equivalence of these definitions of $B_1^1$, or for the following formulation of Peller's theorem~: there is a constant $C>0$ such that for every Hankel matrix $A=(a_{j+k})_{j,k\geq 0}$,
\begin{equation}\label{eq=peller} C^{-1} \| \sum_{n \geq 0} a_n z^n\|_{B_1^1} \leq \|A\|_{S^1} \leq C \| \sum_{n \geq 0} a_n z^n\|_{B_1^1}.\end{equation}

For a function $f\colon [0,\infty) \to \R$ and a subinterval $I$ of $[0,\infty)$ we adopt the following notation
\begin{equation}\label{eq=def_l2} \|f \|_{L^2(I)} = \left(\int_I |f(x)|^2 dx\right)^{\frac 1 2}, \|f \|_{\ell^2(I)} = \left(\sum_{k \in I \cap \mathbf N} |f(k)|^2\right)^{\frac 1 2}.\end{equation}
We will prove the following upper estimate on the $B_1^1$-norm of a function with smooth symbol.

\begin{prop}\label{prop=Hankel_with_C^1_symbol} Let $f\colon [0,\infty) \to \R$ a continuous function of class $C^1$ on $(1,\infty)$, and $\frac12\geq\alpha >0$. Then  
\[ \| \sum_{n \geq 0} f(n) z^n\|_{B_1^1} \leq C\left(|f(0)|+\frac {1}{\sqrt{ \alpha}} (\widetilde A+\sqrt{\widetilde A \widetilde B})\right),\]
for a universal constant $C$, where 
\[\widetilde A = \sqrt{\| x^{\frac 1 2 - \alpha} f\|_{\ell^2([1,\infty))} \| x^{\frac 1 2 + \alpha} f\|_{\ell^2([1,\infty))}},\]
\[\widetilde B = \sqrt{\| x^{\frac 3 2 - \alpha} f'\|_{L^2(1,\infty)} \| x^{\frac 3 2 + \alpha} f'\|_{L^2(1,\infty)}}.\]
\end{prop}
Before we prove the Proposition, we explain how it implies Theorem \ref{R+}.

\begin{proof}[Proof of Theorem \ref{R+}]
Let $f$ be as in Theorem \ref{R+}. Note that both $ A$ and $ B$ are unchanged if the function $f$ is replaced by $x\mapsto f(t x)$. We can therefore restrict ourselves to the case $t=1$.

We first prove the inequalities
\begin{equation}\label{eq=discrete_continuous} \left( \sum_{n\geq1} n^{2\beta} |f(n+1)- f(n)|^2\right)^{\frac 1 2} \leq \|x^\beta f'\|_{L^2(1,\infty)}\end{equation}
for $\beta = \frac 1 2+\alpha$ and $\beta = \frac 1 2-\alpha$, and 
\begin{eqnarray}\label{eq=continuous} \|x^\beta (f'(x+1)-f'(x))\|_{L^2(1,\infty)} 
&\leq \|x^\beta f''\|_{L^2(1,\infty)}
\end{eqnarray}
for $\beta = \frac 3 2+\alpha$ and $\beta = \frac 3 2-\alpha$. Together with \eqref{eq=peller} and Proposition \ref{prop=Hankel_with_C^1_symbol}, they will imply that
\begin{equation}\label{eq=intermediate_ineq}
 \left\|\left(f(j+k)-f(j+k+1)\right)_{j,k \geq 0}\right\|_{S^1}
 \leq C\left( |f(0)-f( 1)|+\frac {1 }{\sqrt{ \alpha}}( A + \sqrt{ A  B} )\right).
\end{equation}

For \eqref{eq=discrete_continuous}, note that $n^\beta \leq x^\beta$ for every integer $n$ and $x \in [n,n+1]$. By Cauchy-Schwarz inequality we have 
\[ n^{\beta} |f(n+1)- f(n)| \leq n^\beta \|f'\|_{L^2(n,n+1)} \leq \|x^\beta f'\|_{L^2(n,n+1)}.\]
By taking the square and summing for $n \geq 1$ we get \eqref{eq=discrete_continuous}. For \eqref{eq=continuous} write $f'(x+1)-f'(x) = \int_{0}^1 f''(x+s) ds$ and use the triangle inequality to get
\[ \|x^\beta (f'(x+1)-f'(x))\|_{L^2(1,\infty)} \leq \int_0^1 \|x^\beta f''(x+s)\|_{L^2(1,\infty)}ds,\]
from which \eqref{eq=continuous} follows because $\|x^\beta f''(x+s)\|_{L^2(1,\infty)} \leq \|x^\beta f''(x)\|_{L^2(1,\infty)}$ for all $0<s<1$.

We now prove the theorem. If $ B = \infty$ there is nothing to prove. So let us assume that $t=1$ and $ B<\infty$. Theorem \ref{R+} follows from \eqref{eq=intermediate_ineq} and the inequalities 
\[ |f(0) - f(1)| \leq \|f'\|_{L^1(\R_+)} \leq \frac{\sqrt 2}{\sqrt\alpha} A\]
and 
\[  A \leq \frac{1}{\sqrt{1-\alpha^2}}  B.\]
To prove the first inequality, decompose the integral and use the Cauchy-Schwarz inequality
\begin{eqnarray*} \|f'\|_{L^1} &= &\int_0^s \frac{ x^{\frac 1 2 - \alpha} f'(x)}{x^{\frac 1 2 - \alpha}} dx  +  \int_s^\infty \frac{f'(x) x^{\frac 1 2 + \alpha}}{x^{\frac 3 2 + \alpha}} dx \\ & \leq& \frac{ s^\alpha}{\sqrt{2\alpha}} \| x^{\frac 1 2 - \alpha} f'\|_{L^2(\R_+)}  + \frac{ s^{-\alpha}}{\sqrt{2\alpha}} \| x^{\frac 1 2 - \alpha} f'\|_{L^2(\R_+)}.\end{eqnarray*}
Taking the infimum over $s>0$ we get $\|f'\|_{L^1(\R_+)} \leq \frac{\sqrt 2}{\sqrt\alpha} A$ as claimed.

Let us move to the inequality $ A \leq \frac{1}{\sqrt{1-\alpha^2}}  B$. By the assumption $ B<\infty$, we have that $f'' \in L^1([1,\infty)$ and hence $\lim_{x \to \infty} f'(x)$ exists. Since $f$ is bounded, this limit is $0$, and we can write $f'(x) = \int_1^\infty x g_r(x) dr$ where $g_r(x) = f''(rx)$. By the triangle inequality
\[ \| x^{\frac 1 2 \pm \alpha} f'\|_{L^2(\R_+)} \leq \int_1^\infty \| x^{\frac 3 2 \pm \alpha} g_r \|_{L^2(\R_+)} dr.\]
By a change of variable 
\[ \| x^{\frac 3 2 \pm \alpha} g_r \|_{L^2(\R_+)} = \frac{1}{r^{2\pm \alpha}} \| x^{\frac 3 2 \pm \alpha} f'' \|_{L^2(\R_+)},\]
and hence using that $\int_1^\infty \frac{dr}{r^{2\pm \alpha}} = \frac{1}{1\pm \alpha}$, we get
\[ \| x^{\frac 1 2 \pm \alpha} f'\|_{L^2(\R_+)} \leq \frac{1}{1\pm \alpha} \| x^{\frac 3 2 \pm \alpha} f'' \|_{L^2(\R_+)}.\]
The inequality  $ A \leq \frac{1}{\sqrt{1-\alpha^2}}  B$ follows.
\end{proof}

We give now the proof of Proposition \ref{prop=Hankel_with_C^1_symbol}. We start with a classical elementary lemma.
\begin{lemma}
\label{lemma=elementary}
If $\varphi \in L^2(\T)$ then
\[ \|\varphi\|_{L^1(\T)} \leq \frac{2}{\sqrt\pi} \sqrt{\|\varphi \|_{L^2(\T)} \| (1-z)\varphi\|_{L^2(\T)}}.\]
\end{lemma}
\begin{proof}
Denote $g(z) =(1-z)\varphi(z)$. For any $0<s<1/2$:
\begin{eqnarray*}
 \|\varphi\|_{L^1} & = & \int_{0}^{1} |\varphi(e^{2i\pi t})| d t\\
 & =&\int_{-s}^s |\varphi(e^{2i\pi t})| d t + \int_{s}^{1-s} \frac{1}{|1-e^{2i\pi
 t}|} |(1-e^{2i\pi
 t})\varphi(e^{2i\pi t})| d t\\
 & \leq & \sqrt{2s} \|\varphi\|_2 + \sqrt{\int_s^{1-s} \frac{1}{|1-e^{2i\pi
  t}|^2} d t} \|g\|_2
\end{eqnarray*}
by the Cauchy-Schwarz inequality. The remaining integral can be computed:
\begin{eqnarray*}
 \int_s^{1-s} \frac{1}{|1-e^{2i\pi
  t}|^2} d t & = &2 \int_s^{1/2} \frac{1}{4\sin^2(\pi t)} dt\\
 & = & \frac{1}{2}\left[\frac{-\cos(\pi t)}{\pi \sin(\pi
  t)}\right]_s^{1/2} = \frac{1}{2 \pi \tan(\pi s)} \leq
 \frac{1}{2\pi^2 s}
\end{eqnarray*}
where we used that $\tan x \geq x$ for all $0\leq x \leq \frac \pi 2$.
Taking $s= \|g\|_2/ 2\pi\|\varphi\|_2 \leq 1/2$ we get the desired inequality.
\end{proof}

Let $f$ be as in  Proposition \ref{prop=Hankel_with_C^1_symbol}. Recall the notation introduced in \eqref{eq=def_l2}. We prove the following.

\begin{lemma}\label{lemma=domination_L1norm} Let $I_n = (2^{n-1},2^{n+1}]$. Denote $\varphi(z) = \sum_{n\geq 0} f(n) z^n$. Then for $n \geq 1$
\[ 2^n\|W_n \ast \varphi\|_{L^1(\T)} \leq \frac{4}{\sqrt \pi}(\|x^{\frac 1 2} f\|_{\ell^2(I_n)} + \sqrt{\|x^{\frac 3 2} f'\|_{L^2(I_n)} \|x^{\frac 1 2} f\|_{\ell^2(I_n)}}).\]
\end{lemma}
\begin{proof}
The inequality $\| W_n \ast \varphi\|_{L^2(\T)} \leq \|f\|_{\ell^2(I_n)}$ is clear. Writing $(1-z)(W_n \ast \varphi)(z)$ as
\[
\sum_{2^{n-1}+1}^{2^{n+1}} (\widehat W_n(k)-\widehat W_n(k-1)) f(k) z^k + \widehat W_n(k-1)( f(k)-f(k-1)) z^k,\]
and noting that $|\widehat W_n(k)-\widehat W_n(k-1)| \leq 2^{1-n}$ and $\widehat W_n(k-1) |f(k)-f(k-1)| \leq \|f'\|_{L^2(k-1,k)}$ for $k\in I_n$, we get
\[\|(1-z) (W_n \ast \varphi)\|_{L^2(\T)} \leq 2^{1-n} \|f\|_{\ell^2(I_n)} + \|f'\|_{L^2(I_n)}.\]
By Lemma \ref{lemma=elementary} and the inequality $\sqrt{a + b} \leq \sqrt a + \sqrt b$ we get
\[ \|W_n \ast \varphi\|_{L^1(\T)} \leq \frac{2^{\frac{3-n}{2}}}{\sqrt \pi} \|f\|_{\ell^2(I_n)} + \frac{2}{\sqrt \pi} \sqrt{\|f\|_{\ell^2(I_n)} \|f'\|_{L^2(I_n)}}.\]
Multiplying by $2^n$ and using that $x \geq 2^{n-1}$ on $I_n$ we get 

\[ 2^n \|W_n \ast \varphi\|_{L^1(\T)} \leq \frac{4}{\sqrt \pi} \left(\| x^{\frac 1 2} f\|_{\ell^2(I_n)} + \sqrt{ \|x^{\frac 1 2} f\|_{\ell^2(I_n)} \|x^{\frac 3 2} f'\|_{L^2(I_n)}}\right),\]
which concludes the proof.
\end{proof}
 
\begin{proof}[Proof of Proposition \ref{prop=Hankel_with_C^1_symbol}]
 By Lemma \ref{lemma=domination_L1norm} and Peller's characterization, there is a universal constant $C$ such that we have 
\[ \|A\|_1 \leq C \left( |f(0)|+|f(1)|+\sum_{n \geq 1} \|x^{\frac 1 2} f\|_{\ell^2(I_n)} + \sqrt{ \|x^{\frac 1 2} f\|_{\ell^2(I_n)} \|x^{\frac 3 2} f'\|_{L^2(I_n)}}\right).\]
Denote here $I_0=\{1\}$, so that $|f(1)| = \|f\|_{\ell^2(I_0)}$. Then by Cauchy-Schwarz inequality the previous inequality becomes
\begin{multline}\label{eq=after_Peller} \|A\|_1 \leq C \left(|f(0)|+\sum_{n \geq 0} \|x^{\frac 1 2} f\|_{\ell^2(I_n)} \right. \\ \left.+ \sqrt{ \sum_{n \geq 1} \|x^{\frac 1 2} f\|_{\ell^2(I_n)}}\sqrt{\sum_{n \geq 1} \|x^{\frac 3 2} f'\|_{L^2(I_n)}}\right).\end{multline}
Let $N \geq 1$. If $n \leq N$, $x^{1/2}$ is dominated by $x^{\frac 1 2 -\alpha} 2^{\alpha(n+1)}$ on $I_n$, and hence
\begin{eqnarray*} \sum_{n = 0}^N \|x^{\frac 1 2} f\|_{\ell^2(I_n)} &\leq& \left(\sum_{n = 0}^N 2^{2\alpha(n+1)}\right)^{\frac 1 2} \left( \sum_{n=0}^N \| x^{\frac 1 2 - \alpha}f\|_{\ell^2(I_n)}^2\right)^{\frac 1 2}\\
& \leq & \frac{2^{\alpha(N+2)}}{\sqrt{2^{2\alpha} - 1}} \sqrt 2 \| x^{\frac 1 2 - \alpha} f\|_{\ell^2([1,2^{N+1}])} \\
& \leq& \frac{C}{\sqrt{\alpha}} 2^{\alpha N} \| x^{\frac 1 2 - \alpha} f\|_{\ell^2([1,\infty))}.
\end{eqnarray*}
The $\sqrt 2$ is because every point in $[1,\infty)$ belongs to at most $2$ intervals $I_n$ for $n \in [1,N]$. For $n > N$ use that $x \geq 2^{n-1}$ on $I_n$ to dominate
\begin{eqnarray*} \sum_{n > N} \|x^{\frac 1 2} f\|_{\ell^2(I_n)} &\leq& \sum_{n > N} 2^{-\alpha(n-1)} \| x^{\frac 1 2 + \alpha} f\|_{\ell^2(I_n)} \\
& \leq &\left(\sum_{n > N} 2^{-2\alpha(n-1)} \right)^{\frac 1 2} \sqrt{2} \|x^{\frac 1 2 + \alpha} f\|_{\ell^2([1,\infty))}\\
& \leq & \frac{C}{\sqrt{\alpha}} 2^{-\alpha N} \|x^{\frac 1 2+\alpha}f\|_{\ell^2([1,\infty))}.
\end{eqnarray*}
Let $a = \| x^{\frac 1 2 - \alpha}f\|_{\ell^2([1,\infty))}$ and $b = \|x^{\frac 1 2+\alpha}f\|_{\ell^2([1,\infty))}$. Since $a \leq b$ we have $\inf_{N \geq 1} 2^{\alpha N} a + 2^{-\alpha N} b \leq 2^{1+\alpha} \sqrt{ab}$. This implies that there is a constant $C$ such that 
\[ \sum_{n \geq 0} \|x^{\frac 1 2} f\|_{\ell^2(I_n)} \leq \frac{C}{\sqrt \alpha} \sqrt{\| x^{\frac 1 2 - \alpha}f\|_{\ell^2([1,\infty))} \|x^{\frac 1 2 + \alpha}f\|_{\ell^2([1,\infty))}}.\]
The same argument implies a similar inequality with $f$ replaced by $x f'$ and the norm $\ell^2$ replaced by the norm $L^2$~:
\[ \sum_{n \geq 1} \|x^{\frac 3 2} f'\|_{L^2(I_n)} \leq \frac{C}{\sqrt \alpha} \sqrt{\| x^{\frac 3 2 - \alpha}f'\|_{L^2(1,\infty))} \|x^{\frac 3 2 + \alpha}f\|_{L^2(1,\infty)}}.\]
If we remember \eqref{eq=after_Peller} we get the inequality in the Proposition, which concludes the proof.\end{proof}

\section{The case of $\Z^d$}\label{section=Zd}
In this section we prove that $\Z^d$ equipped with its standard generating set satisfies the conclusion of Corollary \ref{hyperbolicgroup} for all $d\geq 1$. Actually we prove a stronger result~: the ``if-part'' of Theorem \ref{HSS} holds for $\Z^d$ (see Remark \ref{rem=difference-1-2}). For the standard generating set, the word-length of $n =(n_1,\dots,n_d) \in \Z^d$ is the $\ell^1$-length $|n|=|n_1|+\dots+|n_d|$. 
\begin{thm}\label{Zd} Let $d \geq 1$. There exists $C_d \in \R_+$ such that for every function $\dot \phi \colon \N \to \C$ satisfying that the matrix $H = \left(\dot \phi(j+k) - \dot \phi(j+k+2)\right)$ is trace class, the map
\[ \sum_{n \in \Z^d} c_n e^{i n \cdot x} \mapsto \sum_{n \in \Z^d} c_n \dot \phi(|n|) e^{in \cdot x}\]
is bounded on $L^\infty(\T^d)$ with norm less than $ \lim_{n\rightarrow\infty}(| \dot{\phi}(2n)|+|\dot{\phi}(2n+1)|)+C_d\|H\|_{S^1}$.
\end{thm}



\begin{rem}
This is indeed an analogous of Corollary \ref{hyperbolicgroup} because, as $\Z^d$ is commutative, the von Neumann algebra of $\Z^d$ is $L^\infty(\R^d/\Z^d)$, and the norm and cb norm of a Fourier multiplier on $L^\infty(\R^d/\Z^d)$ coincide. 
\end{rem}
\begin{rem}
 Theorem \ref{Zd} together with  Theorem \ref{HSS} tell us  that the Banach space of c.b. radial multipliers on $\mathbb{F}_d$  embeds naturally into the Banach space of c.b. radial multipliers on $\Z^d$. It is tempting to expect a direct proof of this. We were only able to find a proof that relies on Theorem \ref{HSS} and on estimates for the norm in $L^1(\R^d/\Z^d)$ of functions with radial Fourier transform.
\end{rem}

For the proof of Theorem \ref{Zd}, we can restrict ourselves to the case when $\dot \phi$ has finite support. For a function $\phi \colon \Z^d \to \C$ with finite support, the Fourier multiplier $m_\phi$ with symbol $\phi$ is the convolution by the function $x \in \R^d/\Z^d \mapsto \sum_{n \in \Z^d} \phi(n) e^{2i\pi n\cdot x}$ on $L^\infty(\R^d/\Z^d)$. The norm and cb norm of $m_\phi$ both coincide with the $L^1$-norm of the function $\sum_{n \in \Z^d} \phi(n) e^{2i\pi n\cdot x}$. Taking into account Peller's Theorem \cite{P80} (see \S \ref{subsection=proof}), we see that Theorem \ref{Zd} is equivalent to the existence of $C_d$ such that 
\begin{equation}\label{eq=validity_Zd} \| \sum_{n \in \Z^d} a_{|n|}  e^{ i n \cdot x}\|_{L^1([-\pi,\pi]^d)} \leq C_d \| \sum_{m \geq 0} (a_{m} - a_{m+2}) e^{i m\theta}\|_{B_1^1}\end{equation}
for all finitely supported sequence $(a_m)_{m \geq 0}$. But it is easy to see from \eqref{eq:def_B11} that for $\varphi(e^{i\theta}) = \sum_{m \geq 0} b_m e^{i m\theta}$,
\[ C^{-1} \|\varphi\|_{B_1^1} \leq \sum_{n \geq 0} \|W_n \ast \psi\|_{L^1(\T)} \leq C \|\varphi\|_{B_1^1}\]
for some universal constant $C$, where $\psi(e^{i\theta}) = \sum_{m \geq 0} (m+1) b_m e^{i m\theta}$. 

The inequality \eqref{eq=validity_Zd} therefore follows from 
\begin{prop}\label{prop=inequZd} Let $d$ be an integer. There is a constant $C_d$ such that for every finitely supported sequence $(a_n)_{n\geq 0}$, 
\begin{equation}\label{eq=L1_Dirichlet} \| \sum_{n \in \Z^d} a_{|n|}  e^{ i n \cdot x}\|_{L^1([-\pi,\pi]^d)} \leq C_d \| \sum_{m \geq 0} (m+1) (a_{m} - a_{m+2}) e^{i m\theta}\|_{L^1([-\pi,\pi])}.\end{equation}
\end{prop}
\begin{proof}
We can assume that $d$ is even, because \eqref{eq=L1_Dirichlet} for $d$ implies \eqref{eq=L1_Dirichlet} for $d-1$ by taking the average with respect to $x_d$.

We can rewrite
\[ \sum_{n \in \Z^d} a_{|n|} e^{i n \cdot x} = \sum_{m\geq 0} (a_{m} - a_{m+1}) D_m(x)\]
where $D_m(x) = \sum_{|n|\leq m} e^{i n \cdot x}$. The exact value for $D_m(x)$ was computed in \cite[Theorem 4.2.3]{X95} and is equal to 
\[ D_m(x_1,\dots,x_d) = [\cos x_1,\dots,\cos x_d]G_m\]
where $G_m\colon [-1,1] \to \R$ is given by
\[ G_m(\cos \theta) = (-1)^{\frac d 2 - 1} (\sin \theta)^{d-2} (\cos(m\theta) + \cos((m+1)\theta))\]
and for a function $f \colon [-1,1] \to \C$ and $d$ distinct numbers $t_1,\dots,t_d \in [-1,1]$ we use the following notation of divided difference
\[ [t_1,\dots,t_d]f = \sum_{j=1}^d \frac{f(t_j)}{\prod_{k \neq j} (t_j - t_k)}.\]
Here we use that $d$ is even, otherwise in the formula for $G_m(\cos \theta)$ the terms $\cos$ have to be replaced by $\sin$.

For a function $f \colon [0,\pi] \to \C$ we define $H_f \colon [-1,1] \to \C$ by 
\[H_f(\cos \theta) = (\sin\theta)^{d-2} f(\theta) \textrm{ for }\theta \in [0,\pi].\]
Then we have the identity 
\[ \sum_{m \geq 0} (a_m - a_{m+1}) G_m(\cos \theta) = \frac{(-1)^{\frac d 2 - 1}}{2} H_{f_1+f_2}(\cos \theta),\]
 for all $\theta\in [-\pi,\pi]$ with $f_2(\theta)=f_1(-\theta)$  and  
\[f_1(\theta) = \sum_{m \geq 0} (a_m - a_{m+1})(e^{im\theta} + e^{i(m+1)\theta}) = (a_0-a_1) + \sum_{m\geq 0}(a_m - a_{m+2}) e^{i(m+1)\theta}.\] By the preceding we can therefore write
\[ \sum_{n \in \Z^d} a_{|n|} e^{i n \cdot x} =\frac{(-1)^{\frac d 2 - 1}}{2} [\cos x_1,\dots,\cos x_d]H_{f_1+f_2}.\]
Using that $H_1(t) = (1-t^2)^{\frac d 2 - 1}$ is a polynomial of degree $d-2$ ($d$ is even) and that $[t_1,\dots,t_d]f=0$ whenever $f$ is a polynomial of degree $d-2$ we observe for further use that
\begin{equation}\label{eq=divided_difference_vanish} [\cos x_1,\dots,\cos x_d]H_1 = 0.\end{equation}

We claim that
\begin{equation}\label{eq=divided_difference_ineq} \sup_{0\leq t\leq\pi} \| [\cos x_1,\dots,\cos x_d]H_{\chi_{[t,\pi]}} \|_{L^1([-\pi,\pi]^d)} <\infty \end{equation}
This will imply the proposition. Indeed, if $K$ is the $\sup$ in the previous inequality, and $f \colon [0,\pi] \to \C$ is any $C^1$ function, writing $f(\theta) = f(0) + \int_{0}^{\pi} f'(t) \chi_{[t,\pi]}(\theta) dt$ for all $\theta \in [0,\pi]$ and using \eqref{eq=divided_difference_vanish}, we get
\begin{multline*} \| [\cos x_1,\dots,\cos x_d]H_{f} \|_{L^1([-\pi,\pi]^d)} \leq \\
\int_{0}^\pi |f'(t)| \| [\cos x_1,\dots,\cos x_d]H_{\chi_{[t,\pi]}} \|_{L^1([-\pi,\pi]^d)} dt \leq K \|f'\|_{L^1([0,\pi])}.\end{multline*}
Applying this inequality to $f=f_1+f_2$ and noticing that 
\begin{eqnarray}\label{f1'}
f_1'&=&i\sum_{m\geq 0} (m+1)(a_{m} - a_{m+2}) e^{i(m+1)\theta}\end{eqnarray}
we get  \eqref{eq=L1_Dirichlet}.
We now move to \eqref{eq=divided_difference_ineq}.  Note $H_{\chi_{[t,\pi]}}+H_{\chi_{[0,t]}}=H_1=0$, we have
\[ \| [\cos x_1,\dots,\cos x_d]H_{\chi_{[t,\pi]}} \|_{L^1([-\pi,\pi]^d)} = \| [\cos x_1,\dots,\cos x_d]H_{\chi_{[0,t]}} \|_{L^1([-\pi,\pi]^d)}.\]
Also by symmetry we have 
\[ \| [\cos x_1,\dots,\cos x_d]H_{\chi_{[t,\pi]}} \|_{L^1([-\pi,\pi]^d)} = 2^d \| [\cos x_1,\dots,\cos x_d]H_{\chi_{[t,\pi]}} \|_{L^1([0,\pi]^d)}.\]
Finally, by the change of variables $x_i \mapsto \pi - x_i$,
\[ \| [\cos x_1,\dots,\cos x_d]H_{\chi_{[t,\pi]}} \|_{L^1([0,\pi]^d)} = \| [\cos x_1,\dots,\cos x_d]H_{\chi_{[0,\pi-t]}} \|_{L^1([0,\pi]^d)},\]
so we are left to prove
\[ \sup_{0\leq t \leq \pi/2} \| [\cos x_1,\dots,\cos x_d]H_{\chi_{[0,t]}} \|_{L^1([0,\pi]^d)} <\infty.\]
This will follow from the estimate in Lemma \ref{lemma=free_diff_lemma} below. If $d=2$ this is exactly the Lemma for $s=0$. If $d >2$, fix $0<t\leq \frac \pi 2$, and for $\theta \in [0,\pi]$ write
\[ H_{\chi_{[0,t]}}(\cos \theta) = \int_0^t (d-2) (\sin u)^{d-3} \cos u \chi_{[u,t]}(\theta) du.\]
With the notation of Lemma \ref{lemma=free_diff_lemma} we have for all $t \in (0,\pi/2)$
\[ \|[\cos x_1,\dots,\cos x_d]H_{\chi_{[0,t]}} \|_{L^1([0,\pi]^d)}  \leq \int_0^t (d-2) (\sin u)^{d-3} \cos u \|A_{u,t}^d\|_{L^1} du,\]
which, by Lemma \ref{lemma=free_diff_lemma}, is less than
\[ C_d t^{2-d} \int_0^t  (d-2)  (\sin u)^{d-2} \cos u du = C_d \left(\frac{\sin t}{t}\right)^{d-2} \leq C_d.\]
This concludes the proof of \eqref{eq=divided_difference_ineq} and of the proposition.
\end{proof}

The previous proof used the following lemma.
\begin{lemma}\label{lemma=free_diff_lemma} Let $d \geq 1$ be an integer. For every $0\leq s<t\leq \pi$, define a function $A^d_{s,t} \colon [0,\pi]^d \to \R$ by
\begin{eqnarray*} A^d_{s,t}(x_1,\dots,x_d) &=& [\cos x_1,\dots,\cos x_d] (\cos \theta \mapsto \chi_{[s,t]}(\theta) )\\
& = & \sum_{i=1}^d \frac{\chi_{[s,t]}(x_i)}{\prod_{j \neq i} (\cos x_i - \cos x_j)}.
\end{eqnarray*}
Then there is a constant $C_d$ such that for all $0\leq s <t \leq \frac \pi 2$.
\[ \|A_{s,t}^d\|_{L^1([0,\pi]^d)} \leq C_d t^{2-d}.\]
\end{lemma}
\begin{proof}
We prove by induction on $d$ that a stronger inequality holds. Namely for all $0<s<t \leq \frac \pi 2$,
\begin{equation}\label{eq=induction_hypothesis}
\|A_{s,t}^d\|_{L^1([0,\pi]^d)} \leq C_d (t-s) \left( \frac 1 s |\log(1-\frac s t)| \right)^{d-1}.
\end{equation}
It is easy to see that $(t-s) \left( \frac 1 s |\log(1-\frac s t)| \right)^{d-1} \leq C'_d t^{2-d}$ for some constant $C'_d$ and all $0<s<t\leq \pi$, so that \eqref{eq=induction_hypothesis} is indeed stronger than the lemma. At least for $s>0$, but the case $s=0$ follows by letting $s \to 0$.

The case $d=1$ is obvious because 
\[ \| A^1_{s,t}\|_{L^1([0,\pi])} = \int_0^\pi \chi_{[s,t]}(\theta) d\theta = t-s.\]

Assume that \eqref{eq=induction_hypothesis} holds for $d \geq 1$, and let $0 < s < t \leq \frac \pi 2$. Throughout the proof we will write $ X \lesssim Y$ when we mean $X \leq C Y$ for some constant allowed to depend on $d$ but not on $s,t$. If $x_1,\dots,x_{d+1} \in [s,t]$ then 
 \[ A^{d+1}_{s,t}(x_1,\dots,x_{d+1}) = [\cos x_1,\dots,\cos x_{d+1}]1=0.\]
By symmetry we therefore have
\[ \|A^{d+1}_{s,t}\|_{L^1([0,\pi]^{d+1})} \leq (d+1) \left(\int_0^s + \int_t^\pi\right) \|A^{d+1}_{s,t}(\cdot,\beta)\|_{L^1([0,\pi]^{d})} d\beta.\]
If $\beta \notin [s,t]$ we can write
\begin{eqnarray*} A_{s,t}^{d+1}(x_1,\dots,x_d,\beta) &=& \sum_{i=1}^d \frac{\chi_{[s,t]}(x_i)}{(\cos x_i - \cos \beta) \prod_{j \neq i} (\cos x_i - \cos x_j)}\\
& = &[\cos x_1,\dots,\cos x_d] \left(\cos \theta \mapsto h_\beta(\theta)\chi_{[s,t]}(\theta)\right),\end{eqnarray*}
where we denote $h_\beta(\theta) = \frac{1}{\cos \theta - \cos \beta}$.

At this point we have to distinguish the cases $\beta<s$ and $\beta>t$. Let us first consider the case $0 \leq \beta<s$. Then for $\theta \in [s,t]$ we write $h_\beta(\theta) = h_\beta(t) - \int_s^t h_\beta'(u) \chi_{[s,u]}(\theta) du$, so that by the triangle inequality we get
\[ |A_{s,t}^{d+1}(x_1,\dots,x_d,\beta)| \leq  |h_\beta(t) A^d_{s,t}(x_1,\dots,x_d)| + \int_s^t |h_\beta'(u) A_{s,u}^d(x_1,\dots,x_d) | du.\]
Integrating with respect to $x_1,\dots,x_d$ we obtain
\[ \|A^{d+1}_{s,t}(\cdot,\beta)\|_{L^1([0,\pi]^{d})} \leq |h_\beta(t)| \|A^{d}_{s,t}\|_{L^1} + \int_s^t |h_\beta'(u)| \|A^{d}_{s,u}\|_{L^1} du.\]
If we use the elementary inequalities $|h_\beta(\theta)| \lesssim \frac{1}{\theta(\theta-\beta)}$ and $|h'_\beta(\theta)| \lesssim \frac{\theta}{\theta^2(\theta-\beta)^2}$ valid for all $0 \leq \beta <\theta \leq \frac \pi 2$, we have $\int_0^s |h_\beta(t)| d\beta \lesssim 
\frac{1}{t}|\log(1-\frac s t)|$ and $\int_0^s |h'_\beta(u)| d\beta \lesssim \frac{s}{u^2(u-s)}$ and the previous inequality together with the induction hypothesis yields after integration
\begin{multline*} \int_0^s \|A^{d+1}_{s,t}(\cdot,\beta)\|_{L^1([0,\pi]^{d})} d\beta \\
\lesssim (t-s) \frac{ |\log(1-\frac s t)|^{d}}{s^{d-1}t} + \int_s^t \frac{s^{2-d}}{u^2} |\log(1-\frac s u)|^{d-1} du. \end{multline*}
With the change of variable $v=1-\frac s u$ the last integral becomes
\[s^{1-d} \int_0^{1-\frac s t} |\log v|^{d-1} dv.\]
One can check that this integral is less than $C (t-s) \left( \frac 1 s |\log(1-\frac s t)| \right)^{d}$. The inequality
\[ \int_0^s \|A^{d+1}_{s,t}(\cdot,\beta)\|_{L^1([0,\pi]^{d})} d\beta \lesssim  (t-s) \left( \frac 1 s |\log(1-\frac s t)|\right)^d \]
follows.

When $\beta\geq t$, we write $h_\beta(\theta) = h_\beta(s) + \int_s^t h_\beta'(u) \chi_{[u,t]}(\theta) du$ 
and by the inequalities $|h_\beta(\theta)| \lesssim \frac{1}{\beta(\beta-\theta)}$ and $|h_\beta'(u)| \lesssim \frac{u}{t^2(\beta-u)^2}$ valid for $u,\theta \leq t$, we get similarly
\[ \int_t^\pi \|A^{d+1}_{s,t}(\cdot,\beta)\|_{L^1([0,\pi]^{d})} d\beta \lesssim \frac{ |\log(1-\frac s t)|}{s} \|A^d_{s,t}\|_{L^1}+ \int_s^t \frac{u}{t^2(t-u)} \|A^d_{u,t}\|_{L^1}.\]
By the induction hypothesis the first term is $\lesssim (t-s)(\frac 1 s |\log(1-s/t)|)^d$, and the second one is less than
\begin{eqnarray*}
 \int_s^t \frac{u^{2-d}}{t^2} |\log(1-u/t)|^{d-1}du& \leq &  \frac t{s^d}\int_s^t  |\log(1-u/t)|^{d-1}du/t\\
 &=&\frac t{s^d}\int_{0}^{1 - \frac s t}  |\log v|^{d-1} dv\\
&\lesssim &(t-s)\left( \frac 1 s |\log(1-\frac s t)|\right)^d
 \end{eqnarray*}
Therefore, 
\[ \int_t^\pi \|A^{d+1}_{s,t}(\cdot,\beta)\|_{L^1([0,\pi]^{d})} d\beta \lesssim (t-s)\left( \frac 1 s |\log(1-\frac s t)|\right)^d .\] This completes the proof of \eqref{eq=induction_hypothesis} for $d+1$. The lemma is proved.
\end{proof}
\begin{rem}
 If we consider $f_1'+f_2'$ in (\ref{f1'}) and notice that 
\begin{eqnarray*}
f_1'+f_2'&=&i\sum_{m> 0} m(a_{m-1} - a_{m+1}) e^{im\theta}-i\sum_{m> 0} m(a_{m-1} - a_{m+1}) e^{-im\theta}\\
&=&i\sum_{m\in {\Z}} |m|(a_{|m|-1} - a_{|m|+1}) e^{im\theta}.
\end{eqnarray*}
We then get 
\begin{equation}\label{fm1tod} \| \sum_{n \in \Z^d} a_{|n|}  e^{ i n \cdot x}\|_{L^1([-\pi,\pi]^d)} \leq C_d \| \sum_{m\in {\Z}} |m|(a_{|m|-1} - a_{|m|+1}) e^{im\theta}\|_{L^1([-\pi,\pi])}\end{equation}
for finitely supported $a$, which says that the Fourier multiplier $e^{in\cdot x}\mapsto a_{|n|}e^{in\cdot x}, n\in {\Z}^d$ is bounded   on $L^\infty([-\pi,\pi]^d)$ for all $d\in {\N}$ provided $\lim_{k\rightarrow \infty} |a_{2k}|+|a_{2k+1}|<\infty$ and  the Fourier multiplier $e^{im\theta}\mapsto b_me^{im\theta}$ with $b_m=|m|(a_{|m|-1} - a_{|m|+1}),m\in{\Z}$, is  bounded   on $L^\infty([-\pi,\pi])$.
\end{rem}
\section{BMO and $H^\infty$ Calculus}\label{section=motivation}

A motivation in studying   $S_t^r$ comes from harmonic analysis on free groups. We will  briefly explain it in this section. We will also show a related result on  bounded  $H^\infty$-calculus.

Following \cite{M08} and \cite{JM12}, we may consider BMO spaces associated with the semigroups $S_t^r$ on the free group von Neumann algebras. For $f\in L^2(\hat {\mathbb F}_n)$, let 
$$\|f\|_{BMO^r}=\sup_{t\geq 0} \|S_t^r |f-S_t^r f|^2\|^\frac12.$$
Set 
$$BMO^r(\hat{\mathbb F}_n)=\{f\in L^2, \max\{\|f\|_{BMO^r},\|f^*\|_{BMO^r}\} <\infty\}.$$
Theorem 5.2 of \cite{JM12} says that the complex interpolation space between BMO$^r$ and $L^2(\hat {\mathbb F}_n)$ is $L^p(\hat {\mathbb F}_n)$,
for all $2<p<\infty$ and $0<r\leq1$.
Thus, for any $0<r\leq 1$, BMO$^r$ serves as an endpoint for $L^p(\hat {\mathbb F}_n)$ corresponds to $p=\infty$. 
What will be an endpoint space which could substitute $L^1(\hat {\mathbb F}_n) $? A natural candidate would be the $H^1$ space defined by the Littlewood-Paley G-function 
$$G(f)=(\int_0^\infty |\partial_t S^1_t f|^2tdt)^\frac12$$ for $f\in L^1$ and
$$ \|f\|_{H^1}=\tau G(f)<\infty .$$
In fact, for $n=1$, we have Fefferman--Stein's famous duality $(H^1)^*=BMO^1$  and the corresponding interpolation  result.
 There has not been an satisfactory $H^1$-BMO duality theory associated with semigroups on free group von Neumann algebra for $n>1$. A main obstacle is due to the missing of geometric/metric tools in the noncommutative setting. For example, when $n=1$, all the $H^1$-BMO duality-arguments (to the best knowledge of the authors) rely on an equivalent characterization of $H^1$ by the Lusin area-function, the definition of which is similar to the Littlewood-Paley G-function but uses an integration on cones instead of the radial integration. 
The concept of ``cones" on $\hat{\mathbb F}_n$ is an big mystery for $n>1$. However, there is a semigroup-representation of Lusin-area integrations as follows 
$$Af=(\int_0^\infty S_{t^2}^2|\partial_t S^1_t f|^2tdt)^\frac12,$$ which uses the semigroup $S_t^2$   to  compensate the ``integration on cones"  and $$\|f\|_{H^1}\simeq \| A(f)\|_{L^1}$$ for $n=1$ (see \cite{M08} for an explanation). We should point out that the equivalence $\|f\|_{H^1}\simeq \| A(f)\|_{L^1}$ fails if we replace the extra 
$S_t^2$ in the definition of $Af$ by $S_t^1$. 
The complete boundedness of $S_t^r$, especially for $r=2$, then draws our attention and is proved in Section 3. 
We still do not know wether a semigroup $H^1$-BMO duality holds on ${\cal L}({\mathbb F}_n)$ for $n>1$ and leave the question for later.

Junge--Le Merdy--Xu (\cite{JMX06}) studied $H^\infty$-calculus in the noncommutative setting (see \cite{CDMY96}). In particular, they obtain a bounded $H^\infty$-calculus property of ${\cal L}^r:\lambda_g\mapsto |g|^r\lambda_g$   on $L^p(\hat {\mathbb F}_n)$ and consequently a Littlewood-Paley theory for the corresponding semigroup $S_t^r$ for all $1<p<\infty,0<r\leq 1$.
The end point cases ($p=1,\infty$) are  more subtle and ${\cal L}^r$ has no bounded $H^\infty$-calculus on the group von Neuman algebra ${\cal L}{\mathbb F}_n$. 
In the rest part of this section, we will show that   $S_t^r, 0<r<1$ has  a bounded $H^\infty$-calculus on $BMO^\frac12({\mathbb F}_n)$. 
\begin{prop}\label{Mc}
Suppose $T$ is a sectorial operator on a Banach space $X$. Assume 
$\int_0^\infty Te^{-tT}a(t)dt$ is bounded on $X$ with norm smaller than $C$ for any choice $a(t)=\pm1$.
Then $T$ has a bound $H^\infty(S_\eta^0)$ calculus for any $\eta>\pi/2$. 
\end{prop}
\begin{proof} This  is a consequence of  Example 4.8   of \cite{CDMY96} by setting $a(t)$ to be the sign of $  \langle Te^{-tT}u, v\rangle$ for any pair $(u,v)\in( X,X^*)$.
\end{proof}
\begin{prop}\label{JM12}
Suppose   $a(t)$ is a function on $(0,\infty)$ satisfying 
\begin{eqnarray}\label{correction}
s\int_s^\infty\frac{ |a(t-s)|^2}{t^2}dt\leq c_a^2.\end{eqnarray} 
for any $s>0$.
Then $\int_0^\infty  {\cal L}^\frac12 e^{-t{\cal L}^\frac12}a(t)dt$ is completely bounded on $BMO^\frac12({\mathbb F}_n)$ with upper bound $\lesssim c_a$.
\end{prop}

\begin{proof} We apply Corollary 3.4 of \cite{JM12} to  $S_t^1$. Note that the   subordinated Poisson semigroup of $S_t^1$ is $S_t^{\frac12}$. So  the  space $BMO({\cal P})$ associated with $S_t^1$ as defined   in  \cite{JM12} is  the space $BMO^\frac12$ defined in this section. Corollary 3.4 of \cite{JM12} then implies Proposition \ref{Mc}.
Because the required $\Gamma^2\geq 0$ condition associated with $S_t^1$
is actually the positive definiteness  of   the kernel $K(g,h)=(\frac {|g|+|h|-|g^{-1}h|}2)^2$  on ${\mathbb F}_n\times {\mathbb F}_n$, which easily follows from the negative definiteness of the length function $|\cdot|$.
\end{proof}

\begin{rem} There are a few misprints in \cite{JM12}. The condition of $a(t)$ on page 710 is miss-written. The correct one is (\ref{correction}) in this article. In Thereom 3.3 of \cite{JM12}, the integer $n$ must be  strictly positive. 
\end{rem}

\begin{thm}
For $0<r<1$, ${\cal L}^r:\lambda_g\mapsto |g|^r\lambda_g$ has a bounded $H^\infty(S_\eta^0)$ calculus on $BMO^\frac12(\hat{\mathbb F}_n)$ for any $\eta>r \pi$. 
\end{thm}

\begin{proof} It is easy to see that $S_t^1$ is a bounded semigroup on $BMO^\frac12$. So ${\cal L}^r$ is a sectorial operator on $BMO^\frac12$ of type $\frac{r\pi}2$ for $0<r<1$.
Applying Proposition \ref{JM12} to Proposition \ref{Mc}, for $|a(t)|=1$ and $T={\cal L}^{\frac12}$ we conclude that ${\cal L}^\frac12$ has a bounded $H^\infty(S_\eta^0)$ calculus on $BMO^\frac12$ for any $\eta>\pi/2$. Therefore, ${\cal L}^r$ has a bounded $H^\infty(S_\eta^0)$ calculus on $BMO^\frac12$ for any $\eta>r\pi$. 
\end{proof}

\bigskip
{\bf Acknowlegement.} The authors thank Narutaka Ozawa for helpful comments. 


\bibliographystyle{amsplain}

\begin{thebibliography}{99}

\bibitem{B86} M. Bo\.zejko, Positive definite functions on the free group and the noncommutative Riesz product. (Italian summary) 
Boll. Un. Mat. Ital. A (6) 5 (1986), no. 1, 13-21. 

\bibitem{BF84} M. Bo\.zejko,G. Fendler, Herz-Schur multipliers and completely bounded multipliers of the Fourier algebra of a locally compact group, Boll. Un. Mat. Ital. A(6), 3(2) (1984) 297--302. 

\bibitem{BN08} N. P. Brown, N. Ozawa, $C^*$-algebras and finite-dimensional approximations. 
Graduate Studies in Mathematics, 88. American Mathematical Society, Providence, RI, 2008. xvi+509 pp. 

\bibitem{CDMY96} M. Cowling, I. Doust, A. McIntosh, A. Yagi, Banach space operators with a bounded $H^\infty$-functional calculus. J. Austral. Math. Soc. Ser. A 60 (1996), no. 1, 51-89.

\bibitem{MF80} L.  De-Michele, A. Fig\'a-Talamanca,
Positive definite functions on free groups. 
Amer. J. Math. 102 (1980), no. 3, 503-509. 

\bibitem{D13}
S. Deprez, Radial multipliers on arbitrary amalgamated free products of finite von Neumann algebras. (2013) preprint (arXiv:1310.7880).

\bibitem{G08}
L. Grafakos, Classical {F}ourier analysis, second ed. Graduate Texts in Mathematics, 249. Springer, New York (2008).

\bibitem{G14}
L. Grafakos, Modern Fourier analysis. Third edition. Graduate Texts in Mathematics, 250. Springer, New York, 2014.

\bibitem{H79}
U. Haagerup, An example of a non nuclear $C^*$-algebra, which has the metric approximation property, Invent. Math. 50(3) (1979) 279--293.

\bibitem{HM12}
U. Haagerup, S. M{\"o}ller,
Radial multipliers on reduced free products of operator algebras. 
J. Funct. Anal. 263 (2012), no. 8, 2507--2528. 

\bibitem{HSS10}
U. Haagerup, T. Steenstrup, R. Szwarc,
Schur multipliers and spherical functions on homogeneous trees. 
Internat. J. Math. 21 (2010), no. 10, 1337--1382. 


\bibitem{LY13} S. Lu, D. Yan,  Bochner-Riesz means on Euclidean spaces. World Scientific Publishing Co. Pte. Ltd., Hackensack, NJ, 2013. viii+376 pp. ISBN: 978-981-4458-76-4.

\bibitem{JMX06} M. Junge, C. Le Merdy, Q. Xu, $H^\infty$ functional calculus and square functions on noncommutative Lp-spaces. Ast\`{e}risque No. 305 (2006), vi+138 pp. 

\bibitem{JM12} M. Junge, T. Mei, BMO spaces associated with semigroups of operators. Math. Ann. 352 (2012), no. 3, 691--743.

\bibitem{JMP} M. Junge, T. Mei and J. Parcet, Smooth Fourier Multipliers on Group von Neumann Algebras, {\it Geometric Analysis and Functional Analysis}, 24 (2014), 1913-1980.

\bibitem{K14} S. Knudby, Semigroups of {H}erz-{S}chur multipliers, {\it Journal of Functional Analysis}, 266 (2014) 3 1565--1610.

\bibitem{M08} T. Mei, Tent Spaces Associated with Semigroups of Operators, {\it Journal of Functional Analysis}, 255 (2008) 3356-3406.

\bibitem{M14} S. M{\"o}ller, Radial multipliers on amalgamated free products of {${\rm II}_1$}-factors, {\it International Journal of Mathematics}, 25, 1450026 (2014).

\bibitem{O08} N. Ozawa, Weak amenability of hyperbolic groups. Groups
 Geom. Dyn., 2 (2008), 271--280.

\bibitem{P80}V. V. Peller, Hankel operators of class $S_p$ and their applications (rational approximation, Gaussian processes, the problem of majorization of operators), Mat. Sb. (N.S.) 113(155)(4(12)) (1980) 538--581.

\bibitem{P03}
V. V. Peller.
\newblock {\em Hankel operators and their applications}.
\newblock Springer Monographs in Mathematics. Springer-Verlag, New York, 2003.

\bibitem{Pis01} G. Pisier, Similarity Problems and Completely Bounded Maps, expanded version Lecture Notes in Mathematics, Vol. 1618 (Springer-Verlag, Berlin, 2001).

\bibitem{PS86} T. Pytlik, R. Szwarc,  
An analytic family of uniformly bounded representations of free groups. 
Acta Math. 157 (1986), no. 3-4, 287-309. 

\bibitem{W95} J. Wysocza\'nski, A characterization of radial Herz-Schur multipliers on free products of discrete groups, J. Funct. Anal. 129(2) (1995) 268--292.

\bibitem{X95} Y. Xu, Christoffel functions and {F}ourier series for multivariate orthogonal polynomials, J. Approx. Theory 82(2) (1995) 205--239. 

\bibitem{Y80} K. Yosida, Functional analysis. Reprint of the sixth (1980) edition. Classics in Mathematics. Springer-Verlag, Berlin, 1995. xii+501 pp.

\end{thebibliography}

\bigskip

\hfill \noindent \textbf{Tao Mei} \\
\null \hfill Department of Mathematics
\\ \null \hfill Wayne State University \\
\null \hfill 656 W. Kirby Detroit, MI 48202. USA \\
\null \hfill\texttt{mei@wayne.edu}
\bigskip

\hfill \noindent \textbf{Mikael de la Salle} \\
\null \hfill CNRS-ENS de Lyon,
\\ \null \hfill UMPA UMR 5669 \\
\null \hfill F-69364 Lyon cedex 7, France\\
\null \hfill\texttt{mikael.de.la.salle@ens-lyon.fr}

\end{document}